\documentclass[12pt,a4paper,oneside]{article} 
\usepackage{thesis}
\usepackage{rotating,graphicx,amsbsy,amsfonts,epsfig,amsmath}
\usepackage{amssymb,psfrag,lscape,epsf,color,listings}
\usepackage{amsthm}
\usepackage{latexsym,graphicx,psfrag}
\usepackage{a4wide}
\usepackage{enumerate}
\usepackage{comment}
\usepackage{graphicx}
\usepackage{makeidx}
\usepackage{authblk}
\usepackage{hyperref}
\input xy  
\xyoption{all}

\newtheorem{theorem}{Theorem}[section]
\newtheorem{corollary}[theorem]{Corollary}
\newtheorem{lemma}[theorem]{Lemma}

\newtheorem{definition}[theorem]{Definition}

\numberwithin{equation}{section}

\def\NN{{\mathbb N}}

\def\S{\Sigma }

\def\<{\left<}
\def\>{\right>}

\newcommand{\be}{\begin{enumerate}}
\newcommand{\ee}{\end{enumerate}}
\newcommand{\bd}{\begin{description}}
\newcommand{\ed}{\end{description}}

\makeindex

\begin{document}
\title{The additive structure of the squares inside rings}
\author{David Cushing\\ George Stagg}

%\affil{School of Mathematics and Statistics
%\\Newcastle University
%\\Newcastle upon Tyne NE1 7RU
%\\United Kingdom}
\maketitle
\begin{abstract}
When defining the amount of additive structure on a set it is often convenient to consider certain sumsets; Calculating the cardinality of these sumsets can elucidate the set's underlying structure. We begin by investigating finite sets of perfect squares and associated sumsets. We reveal how arithmetic progressions efficiently reduce the cardinality of sumsets and provide estimates for the minimum size, taking advantage of the additive structure that arithmetic progressions provide. We then generalise the problem to arbitrary rings and achieve satisfactory estimates for the case of squares in finite fields of prime order. Finally, for sufficiently small finite fields we computationally calculate the minimum for all prime orders.
\end{abstract}

\section{Introduction}
In \cite{Cr}, problem $4.6,$ T. Tao and B. Green ask how small can $|A+A|$ be for an $n$-element subset $A$ of the set of squares of integers. This question originates from the paper \cite{Ch} where it is shown that $|A+A| > c n(n{\rm ln}(n))^{\frac{1}{12}}$ for some absolute constant $c>0.$ In our paper we obtain upper bounds on how small $|A+A|$ can be. See \cite{Ci} for another paper investigating this problem.
\\
\\
We generalise this problem to arbitrary rings and give satisfactory answers for fields of a prime order. We also show connections between the original problem and the question of the existence of perfect cuboids.

\section{Definitions and notation}
The sumset (also known as the Minkowski sum) of two sets, $A$ and $B$, is defined as the set of all possible results from summing an element of $A$ and an element of $B$. We are primarily concerned with the case when A is a finite set of the natural numbers and when $B=A$.

\noindent Let $A$ be a finite subset of the natural numbers. We can then define the sumset of $A$ as
$$A+A:=\{a+b|a,b\in A\}.$$ 
Define $S(R)$ to be the set of squares in a given ring $R$,
$$S(R):=\{a^{2}|a\in R\}.$$
We aim to minimise the sumset of finite subsets of $S(R)$,
$$N_{n}(R):=\inf_{A\subset S(R),|A|=n}|A+A|,$$
where the size of the subset is $n\in\mathbb{N}$. We will provide upper bounds for $N_{n}(R)$ and compute $N_{n}(\mathbb{Z})$ for sufficiently small n.

\section{Sumsets in integral domains}

We are primarily interested in estimating $N_{n}(\mathbb{Z}).$ However it is easier to work over $\mathbb{Q}$ than $\mathbb{Z}$. In this section we show that
$$N_{n}(\mathbb{Z})=N_{n}(\mathbb{Q}).$$
Furthermore we show that it is possible to generalise this property to integral domains and their field of fractions.
\begin{definition}
Let $R$ be a commutative ring such that for all $a,b\in R$ such that $a\neq0$ and $b\neq 0$ then $a\cdot b\neq 0.$ We say that $R$ is an integral domain.
\end{definition}

\begin{definition}
Let $R$ be an integral domain. For $a,b \in R$ let $\frac{a}{b}$ denote the equivalence class of fractions where $\frac{a}{b}$ is equivalent to $\frac{c}{d}$ if and only if $ad = bc$. The field of fractions of $R$ is the set of all such equivalence classes with the obvious operations.
\end{definition}

\begin{theorem}
Let $R$ be an integral domain. Then there exists a field $F$ such that
$$N_{n}(R)=N_{n}(F)$$
for all $n\in\mathbb{N}.$ Furthermore if $R$ has characteristic $c$ then $F$ can be chosen to have characteristic $c$ also.
\end{theorem}
\begin{proof}
Let $F$ be the field of fractions of $R$. Note that $F$ and $R$ have the same characteristic.
\\
\\
Let $n\in\NN$
\\
\\
Since $R\subset F$ it is clear that 
$$N_{n}(R)\geq N_{n}(F).$$
Let $A\subset S(F)$ such that $|A| = n$ and $|A+A| = N_{n}(F).$ Suppose that
$$A = \left\{\left(\frac{a_{1}}{b_{1}}\right)^{2},\ldots,\left(\frac{a_{n}}{b_{n}}\right)^{2}\right\},$$
where $a_{1},\ldots,a_{n},b_{1},\ldots,b_{n}\in R.$
Consider the set
$$A' = \left\{(a_{1}\times b_{2}\times\cdots\times b_{n})^{2},\ldots, (a_{n}\times b_{1}\times\cdots\times b_{n-1})^{2} \right\}.$$
Then
$$|A'+A'| = N_{n}(F).$$
Thus 
$$N_{n}(R)\leq N_{n}(F).$$
\end{proof}
\begin{lemma}
Let $F$ be a field and $G$ a subfield of $F$. Then
$$N_{n}(G)\geq N_{n}(F)$$
for all $n\in\mathbb{N}.$
\end{lemma}
\begin{proof}
Let $n\in\NN.$ Let $A\subset S(F)$ such that $|A| = n$ and $|A+A| = N_{n}(F).$
\\
\\
Since $A\subset S(G)$ we have that $N_{n}(G)\geq N_{n}(F)$. 
\end{proof}
\begin{corollary}
Let $R$ be an integral domain of characteristic $0$. Then
$$N_{n}(R)\leq N_{n}(\mathbb{Z}).$$
\end{corollary}
\begin{proof}
Let $F$ be a field of characteristic $0$ such that $N_{n}(R)=N_{n}(F)$. Since $F$ has characteristic 0 it contains a subfield isomorphic to $\mathbb{Q}.$ The result easily follows. 
\end{proof}

\section{Properties of the elliptic curve $E_{d}:\:y^{2}=x^{3}-d^2x$}
In this section we use the results of arithmetic progressions of squares from \cite{Co}.
\\
\\
Let $d\in\mathbb{N}.$ Consider the elliptic curve
$$E_{d}:\:y^{2}=x^{3}-d^{2}x.$$
Define 
$$E_{d}[\mathbb{Q}]=\{(x,y)\in E_{d}|x,y\in\mathbb{Q}\}.$$
The set $E_{d}[\mathbb{Q}]$ is deeply connected to locating arithmetic progressions of 3 rational squares. We now demonstrate this.
\begin{lemma}
Let $P=(x,y)\in E_{d}[\mathbb{Q}]$ and set
$$a=\frac{x^{2}-2dx-d^2}{2y},$$
$$b=\frac{x^{2}+d^2}{2y},$$
$$c=\frac{-x^{2}-2dx+d^2}{2y}.$$
Then $\{a^{2},b^{2},c^{2}\}$ is an arithmetic progresion of lentgh 3 in $\mathbb{Q}$ with common difference d.
\end{lemma}

\begin{proof}
Observe that
\begin{align*}
b^{2}-a^{2} & =\left(\frac{x^{2}+d^2}{2y}\right)^{2}-\left(\frac{x^{2}-2dx-d^2}{2y}\right)^{2}
\\
& = \left(\frac{2x^{2}-2dx}{2y}\right)\left(\frac{2d^2+2dx}{2y}\right)
\\
& = \left(\frac{x^{2}-dx}{y}\right)\left(\frac{d^2+dx}{y}\right)
\\
& = \frac{dx^3-d^{3}x}{y^2}
\\
& = \frac{d(x^3-d^{2}x)}{x^3-d^{2}x}
\\
& = d.
\end{align*}
Similarly
\begin{align*}
c^{2}-b^{2} & =\left(\frac{-x^{2}-2dx+d^2}{2y}\right)^{2}-\left(\frac{x^{2}+d^2}{2y}\right)^{2}
\\
& = \left(\frac{2d^{2}-2dx}{2y}\right)\left(\frac{-2x^2-2dx}{2y}\right)
\\
& = \left(\frac{d^{2}-dx}{y}\right)\left(\frac{-x^2-dx}{y}\right)
\\
& = \frac{dx^3-d^{3}x}{y^2}
\\
& = \frac{d(x^3-d^{2}x)}{x^3-d^{2}x}
\\
& = d.
\end{align*}
\end{proof}

\begin{lemma}
Let $a,b,c\in\mathbb{Q}$ such that $\{a^{2},b^{2}c^{2}\}$ is an arithmetic progression of length 3 with common difference d then 
$$\left(\frac{d(c-b)}{a-b},\frac{d^2(2b-a-c)}{(a-b)^{2}}\right)\in E_{d}[\mathbb{Q}].$$
\end{lemma}

\begin{proof}
One simply has to verify that
$$\frac{d^{4}(2b-a-c)^{2}}{(a-b)^{4}}-\frac{d^{3}(c-b)^{3}}{(a-b)^{3}}+\frac{d^{3}(c-b)}{a-b}=0.$$
\end{proof}
We say that the points of $E_{d}[\mathbb{Q}]$ and the arithmetic progressions in the above lemmas are associated to each other.
\\
\\
Let $P=(x,y)\in E_{d}[\mathbb{Q}].$ Define 
$$P\circ P=\left(\left(\frac{x^{2}+d^2}{2y}\right)^{2},Y\right).$$
Where $Y$ is chosen such that  $P\circ P\in E_{d}[\mathbb{Q}].$

\section{Sumsets of squares in $\mathbb{Z}$}
We now turn our attention to estimating $N_{n}(\mathbb{Z})$. Let $N_{n} = N_{n}(\mathbb{Z}) = N_{n}(\mathbb{Q}).$
\\
\\
It is not diffictult to show that
$$2n-1\leq N_{n}\leq \frac{n(n+1)}{2}.$$
This instantly gives that $N_{1}=1$ and $N_{2}=3.$ We also see that 
$$5\leq N_{3}\leq 6.$$
To compute the value of $N_{3}$ let  
$$A=\{1,25,49\}.$$ 
We have
$$A+A=\{1,26,50,64,98\}. $$
In this case $|A+A|=5.$ Which shows that $N_{3}=5.$
\\
\\
 This happened because $\{1,25,49\}$ is arithmetic progression. It turns out that arithmetic progressions are an efficient way to make the sumset small and infact an arithmetic progression is how we obtain the minumum bound of $2|A|-1.$ 
\\
\\
Fermat showed that there exists no airthmetic progression of 4 or more squares so this method does not generalise. 
\\
\\
Although there exists no arithmetic progression of squares of length 4 we do have the following interesting set
$$A=\{49,169,289,529\}.$$
 $A$ is an arithmetic progression of squares of length 5 with the 4th term removed. Since $|A+A|=8$ we obtain that $N_{4}=8$ (because $N_{4}=7$ implies the existence of an arithmetic progression of 4 squares).
\\
\\
Let $n\in\mathbb{N}.$ We will calculate an upperbound for $N_{3n}$ and then state similar results for $N_{3n+1}$ and $N_{3n+2}.$ 
\\
\\
Before we begin note the crude upper bound we already have.
$$N_{3n}\leq \frac{3n(3n+1)}{2}\sim \frac{9}{2}n^{2}.$$
We improve this with
\begin{theorem}\label{ub}
Let $n\in\mathbb{N}.$ Then
$$N_{3n}\leq \frac{5n(n+1)}{2}\sim\frac{5}{2}n^{2}.$$
\end{theorem}
\begin{proof}
We will construct a set $A\subset S(\mathbb{Q}),$ $|A|=3n$ such that 
$$|A+A|\leq \frac{5n(n+1)}{2}.$$
We first prove the case for $n=1.$ Let $A_{1}=\{1,25,49\}.$ Then $|A_{1}+A_{1}|=5.$
\\
\\
Now suppose that $n\geq 2.$ 
\\
\\
Let $P_{1}$ be the point of $E_{24}[\mathbb{Q}]$ associated to $A_{1}.$
\\
\\
Let $P_{i+1}=P_{i}+P_{i}$ for $1\leq i\leq n-1.$ Let $A_{i}$ be the arithmetic progression associated with $P_{i}$ for each $2\leq i \leq n.$
\\
\\
Let 
$$A=\bigcup_{i=1}^{n}A_{i}.$$
We have that $A\subset S(\mathbb{Q})$ and $|A|=3n.$ We claim that 
$$|A+A|\leq\frac{5n(n+1)}{2}.$$
Note that if $B$ and $C$ are two arithmetic progressions of length 3 and the same common difference such that $B\cap C=\emptyset$ then $|B+C|=5.$ 
\\
\\
Therefore
\begin{align*}
|A+A| & = \left|\left(\bigcup_{i=1}^{n}A_{i}\right)+\left(\bigcup_{i=1}^{n}A_{i}\right)\right|
\\
& \leq \sum_{k=1}^{n}|A_{k}+A_{k}|+\sum_{i=1}^{n-1}\sum_{j>i}|A_{i}+A_{j}|
\\
& = \frac{5n(n+1)}{2}.
\end{align*}
\end{proof}
Using a simlar technique we obtain the following result.
\begin{theorem}\label{ub2}
Let $n\in\mathbb{N}.$ Then
$$N_{3n}\leq\frac{5n^{2}+n}{2},$$
$$N_{3n+1}\leq\frac{5n^{2}+9n+2}{2},$$
$$N_{3n+2}\leq\frac{5n^{2}+13n+6}{2}.$$
\end{theorem}
\section{Possible ways for small sumsets}
It is a famous problem as to whether or not a perfect cuboid exists. A perfect cuboid is a cuboid with integer sides, intger faces and an integer valued long diagonal. It is conjectured that such an object does not exist. However if a perfect cuboid does exist then there exists integers $a,b,c,d,e,f,g$ such that
\begin{align*}
a^{2}+b^{2} & =d^{2}
\\
a^{2}+c^{2} & =e^{2}
\\
b^{2}+c^{2} & =f^{2}
\\
a^{2}+b^{2}+c^{2} & =g^{2}.
\end{align*} 
This is a lot of additive structure amongst squares. In fact the existence of a perfect cuboid was equivalent to the existence of a 3x3 magic square such that every entry is a square and to the notion of a generalised arithmetic progression.   
\begin{definition}
Let $a,n_{1},...,n_{d},N_{1},...,N_{d}\in\mathbb{N}.$ Set 
$$A=\{a+m_{1}n_{1}+...+m_{d}n_{d}|0\leq m_{i}\leq N_{i}\}.$$
We say the $A$ is a generalised arithmetic progression. More specifically $A$ is a $(N_{1}-1)\times\ldots\times(N_{d}-1)$ generalised arithmetic progression of dimension $d$. 
\end{definition}
A $3\times 3$ generalised arithmetic progression is equivalent to a 3x3 magice square. Let $A$ be a 3x3 generalised arithmetic progression. Then $|A|=9$ and $|A+A|=25$ Note that this is $5$ lower than our upperbound for $N_{9}$ obtained in Theorem $\ref{ub}.$
\\
\\
The trick to calculating $N_{4}$ was finding 4 elements inside an arithmetic progression of length 5.  Although a 3x3 GAP of squares has not been found a 3x3 Gap with 7 of its elements square has been found.
\begin{center}
{\begin{picture}(180,180)(0,0)
\put(0,0){\line(0,1){180}}
\put(60,0){\line(0,1){180}}
\put(120,0){\line(0,1){180}}
\put(180,0){\line(0,1){180}}
\put(0,0){\line(1,0){180}}
\put(0,60){\line(1,0){180}}
\put(0,120){\line(1,0){180}}
\put(0,180){\line(1,0){180}}
\put(20,150){$373^{2}$}
\put(80,150){$289^2$}
\put(140,150){$565^2$}
\put(20,90){360761}
\put(90,90){$425^2$}
\put(140,90){$23^2$}
\put(20,30){$205^2$}
\put(80,30){$527^2$}
\put(140,30){222121}
\end{picture}}
\end{center}

 This gives us a candidate to lower our upper bound for $N_{7}.$ Theorem \ref{ub2} gives us that $N_{7}\leq 20.$ If we let $A$ be the 7 squares in the above magic square then $|A+A|=19.$ Thus showing that our upperbound is not optimal.

\section{Fields of prime order}
We now turn to our attention to sets of squares inside a finite field. In particular to fields with a prime order. The following inequality gives us a lower bound for $N_{n}(\mathbb{Z}_{p}).$
\begin{theorem}[Cauchy-Davenport inequality]
If p is a prime and $A$ is a set in $\mathbb{Z}_{p}$ then
$$|A+A|\geq min(2|A|-1,p).$$
\end{theorem}
We will show that this bound can be attained for all $n\in\mathbb{N}$ for $p$ sufficiently large. We do this by applying an inverse theorem to the Cauchy-Davenport inequality.
\begin{theorem}[Vosper]
Let $p$ be a prime and $A$ a set in $\mathbb{Z}_{p}$ such that $|A|>2$ and $|A+A|\leq p-2.$ Then $|A+A|=2|A|-1$ if and only if $A$ is an arithmetic progression.
\end{theorem}

\begin{theorem}[Van der Waerden, Gowers]\label{vdw}
 Let $r,k\in\mathbb{N}.$ Then there exists $N(r,k)\in\NN$ such that if $\{1,2,...,N(r,k)\}$ is expressed as the disjoint union of non-empty sets, $\{A_{j}\}_{j=1}^{r},$ then there exists a $j$ such that $A_{j}$ contains an arithmetic progression of length $k$. Furthermore $N(r,k)\leq 2^{2^{r^{2^{2^{9+k}}}}}.$
\end{theorem}
We now prove the main result of this section.
\begin{theorem}
Let $n\in\mathbb{Z}$ and let $p>2^{2^{2^{2^{2^{9+n}}}}}$ be a prime number. Then
$$N_{n}(\mathbb{Z}_{p})=2n-1.$$
\end{theorem}
\begin{proof}
Let $R=S(\mathbb{Z}_{p})\setminus\{0\}$ and $T=\mathbb{Z}_{p}\setminus S(\mathbb{Z}_{p}).$ Then $R$ and $T$ are disjoint non empty subsets of $\{1,2,...,p\}$ such that $R\cup T=\{1,2,...,p\}.$ Therefore, by Van der Waerdens Theorem, either $R$ or $T$ contains an arithmetic progression of length $n$. Suppose that $T$ contains such a progression and denote it by $P.$ Let $p=min{P}.$ Then one can show that $p\cdot P$ is a subset of $R$ and is an arithmetic progression of length $n$. Therefore $R$ contains an arithmetic progression of lentgh $n$ which we denote by $Q$. By Vospers theorem we have
$$|Q+Q|=2n-1.$$
This shows that $N_{n}(\mathbb{Z}_{p})\leq 2n-1.$ By the Cauchy-Davenport inequality we have that $N_{n}(\mathbb{Z}_{p})\geq 2n-1$. Thus completing the proof.    
\end{proof}
Note that the lower bound for $p$ given above is probably much much more than needed. In fact for small values of $n$ we have been able to significantly improve this lower bound. In fact we have the following by applying the results on arithmetic progressions of quadratic residues from \cite{Br}.
\begin{theorem}
$$N_{5}(\mathbb{Z}_{p})=9\:\:for\:\:p\geq 41,$$
$$N_{6}(\mathbb{Z}_{p})=11\:\:for\:\:p\geq 149,$$
$$N_{7}(\mathbb{Z}_{p})=13\:\:for\:\:p\geq 619,$$
$$N_{8}(\mathbb{Z}_{p})=15\:\:for\:\:p\geq 1087,$$
$$N_{9}(\mathbb{Z}_{p})=17\:\:for\:\:p\geq 3391.$$
\end{theorem}

Therefore for small $n$ we could calculate $N_{n}(\mathbb{Z}_{p})$ for all $p$. Below are the values for $n=5$ or $6$ calculated by computer. 
\begin{theorem}
 $N_{5}(\mathbb{Z}_{p})=$ 
$\begin{cases}
 9 & for\:\: p=17,23\:\:and\:\:p\geq 41 \\ 
10 & for\:\: p =11,19,29,31,37\\
11 & for\:\:p=13.
\end{cases}$
\end{theorem}
\begin{theorem}
 $N_{6}(\mathbb{Z}_{p})=$ 
$\begin{cases}
 11 & for\:\: p=11,53,61,67,71,73,79,83,89,97,101,103,\\
&107,109,127,131,137\:\:and\:\:p\geq 149 \\ 
12 & for\:\: p =17,23,41,43,47,113,139\\
13 & for\:\:p=13,19,31,37,59\\
14 & for\:\:p=29.
\end{cases}$
\end{theorem}

\section{Remarks}
There are still many problems in this area that we need to answer. There exists a 2x3 GAP of squares. Does there exists a 3x3 GAP. A 2x2x2 GAP. A 2x..x2 GAP? If we can find arbritrary long 2x..x2 GAPs then we can show that $N_{n}\leq K\cdot x^{2-\varepsilon}$ for some $\varepsilon>0.$ 
\\
\\
What happens for infinite integral domains of finite characteristic. Embedding a field of prime order is very unsatisfactory here as it gives no information for values of $N_{n}(R)$ beyond a certain point.
\\
\\
We need to look more at fields with a prime power number of elements. Consider $\mathbb{F}_{9}.$ Then   there exists a copy of $\mathbb{Z}_{3}\subset\mathbb{F}_{9}$ consisting entirely of squares. Since $\mathbb{Z}_{3}$ is closed under addition we obtain
$$|\mathbb{Z}_{3}+\mathbb{Z}_{3}|=|\mathbb{Z}_{3}|=3.$$
Which gives that $N_{3}(\mathbb{F}_{9})=3<5.$

\end{document}